\documentclass[aps, amsmath, amssymb,  pre,showkeys, twocolumn,groupedaddress,superscriptaddress,nofootinbib,longbibliography]{revtex4-1}

\usepackage{multirow}

\usepackage{enumitem}
\setlist{noitemsep,topsep=0pt,parsep=0pt,partopsep=0pt}
\usepackage{array}

\usepackage[english]{babel}

\usepackage[colorlinks,linkcolor=black,citecolor=blue,urlcolor=black]{hyperref}
\usepackage{amsmath}
\usepackage{tabularx}
\usepackage{amsthm}
\usepackage[T1]{fontenc}
\usepackage{textcomp}
\usepackage{mathptmx}
\usepackage[scaled=.92]{helvet}
\usepackage{courier}
\usepackage{bm}

\usepackage{graphics}
\usepackage{graphicx}
\usepackage{xcolor}
\usepackage[caption=false]{subfig}

\graphicspath{{./FIG/}}
\usepackage{hhline}

\usepackage{lineno}

\modulolinenumbers[5]

\newtheorem{theorem}{Theorem}

\newtheorem{definition}{Definition}
\newtheorem{remark}{Remark}

\begin{document}

\title{New adaptive synchronization algorithm for a general class of hyperchaotic complex-valued systems with unknown parameters and its application to secure communication }
\author{A. A.-H. Shoreh}
\email[]{Corresponding author: ahmed.shoreh@yahoo.com}
\affiliation{Faculty of Mathematics and Mechanics, St. Petersburg State University,
Peterhof, St. Petersburg, Russia}
\affiliation{Department of Mathematics, Faculty of Science, Al-Azhar University, Assiut, 71524, Egypt}
\author{N. V. Kuznetsov}
\affiliation{Faculty of Mathematics and Mechanics, St. Petersburg State University,
Peterhof, St. Petersburg, Russia}
\affiliation{Faculty of Information Technology,
University of Jyv\"{a}skyl\"{a}, Jyv\"{a}skyl\"{a}, Finland}
\affiliation{Institute for Problems in Mechanical Engineering RAS, Russia}
\author{T. N. Mokaev}
\affiliation{Faculty of Mathematics and Mechanics, St. Petersburg State University,
Peterhof, St. Petersburg, Russia}

\date{\today}

\begin{abstract}
This report aims to study adaptive synchronization between a general class of hyperchaotic complex-valued systems with unknown parameters, which is motivated by extensive application areas of this topic in nonlinear sciences (e.g., secure communications, encryption techniques, etc.). Based on the complexity of hyperchaotic dynamical systems, which may be beneficial in secure communications, a scheme to achieve adaptive synchronization of a general class of hyperchaotic complex-valued systems with unknown parameters is proposed. To verify scheme's consistency, the control functions based on adaptive laws of parameters are derived analytically, and the related numerical simulations are performed. The complex-valued Rabinovich system describing the parametric excitation of waves in a magneto-active plasma is considered as an interesting example to study this kind of synchronization. A scheme for secure communication and improving the cryptosystem is proposed; the sketch is constructed to split the message and inject some bit of information signal into parameters modulation and the other bit into the transmitter system's states, which, in turn, complicates decryption task by intruders. Meanwhile, the information signal can be accurately retrieved at the receiver side by adaptive techniques and decryption function. Different types of encrypted messages are considered for testing the robustness of the proposed scheme (e.g., plain text and gray images with diverse scales of white Gaussian noise).

\end{abstract}

\maketitle
\section{Introduction}
Synchronization and control of chaotic dynamical systems are important topics in
applied science owing to their vast application fields in physical systems, networks, image processing, secure communications, stock
markets etc. (see e.g.\cite{Xu-Wang-2014,Tavazoei-Haeri-2008,Asheghan-Miguez-2011,Martin-Poon-2001}). Starting from the pioneering work of Pecora and Carroll \cite{Pecora-Carroll-1990}, in which an effective technique to synchronize two identical chaotic systems with different starting points is proposed, diverse types of chaos synchronization methods have been discovered to synchronize chaotic systems, such as complete synchronization \cite{Lin-He-2005,Mahmoud-Mahmoud-2010}, active control \cite{Zhang- Ma-2004,Farghaly-Shoreh2018,Shoreh-Kuznetsov-2021}, lag synchronization \cite{Du-Zeng-2010,Shoreh-Kuznetsov-2020}, sliding mode control \cite{Tavazoei-Haeri-2008}, cluster synchronization \cite{Tang-Park-2018}, adaptive synchronization \cite{He-Tu-2011,Xua-Zhou-2011,Li-Cao-2008} and many more. Up to date, chaos synchronization has been studied in depth for dynamical systems described by real-valued variables. Meantime, in the applied science there are many problems, such as, e.g., problems of detuned laser optics, or baroclinic instability (see \cite{Ning-Haken-1990,Gibbon-Haken-1982}), which are described by complex-valued dynamical systems. Consideration of such systems in synchronization problems rather than real-valued ones and the related doubling of the number of variables may increase synchronization quality and efficiency, which,in turn, is crucial for many applications, such as, e.g., secure communication and cryptosystems design \cite{Mahmoud-Bountis-2007}.
\par Many chaos synchronization methods are constructed using known parameters of systems, and the corresponding controllers can be easily derived. However, sometimes the parameters of systems in engineering applications may be uncertain in advance and change over time \cite{Mahmoud-2012,Xu-Zhou-2010,Liao-Lin-1999,Hao-2010}. So, it is more sensible to investigate synchronization and develop corresponding methods for such dynamical systems with unknown parameters. In \cite{Mahmoud-2012}, Mahmoud proposed a scheme to achieve adaptive synchronization for a class of hyperchaotic complex-valued systems with uncertain parameters, which can be written in a general form as follows:
\begin{equation}\label{matrix:sys}
  \dot{\textsf{x}}=F(\textsf{x})A+G(\textsf{x}),
\end{equation}
where $\textsf{x}\in \mathbb{C}^{n}$ is a complex state vector, $F\in\mathbb{C}^{n\times n}$, $A$ is a constant vector of system parameters that may be complex or real, $G\in \mathbb{C}^{n}$ is a continuous nonlinear vector function. Another scheme to realize adaptive dual synchronization of the same class \eqref{matrix:sys} with uncertain parameters was proposed in \cite{Mahmoud-Farghaly-2018}. Many complex-valued systems can be written in the form \eqref{matrix:sys}, for instance, the complex-valued Lorenz system \cite{Gibbon-Haken-1982}
,the complex-valued L$\ddot{u}$ system \cite{Mahmoud-Mahmoud-2009}
and so on. Remark that for all systems in the form \eqref{matrix:sys}, the system's parameters associate with linear terms, and only in this case, the schemes in \cite{Mahmoud-2012,Mahmoud-Farghaly-2018} are applicable. Further, there is a class of complex-valued dynamical systems; the parameters are associate with linear and nonlinear terms and can't be written in the form \eqref{matrix:sys}. Therefore, in order to overcome these difficulties in this article we propose a general formula for the complex-valued systems as follows:
\begin{equation}\label{matrix:general:sys}
 \dot{\textsf{x}}=F(\textsf{x})A+G(\textsf{x})+H(\textsf{x})B,
\end{equation}
where $\textsf{x}\in \mathbb{C}^{n}$ is the complex state vector, $F\in\mathbb{C}^{n\times n}$, $A,B$ are constant vectors
of system parameters that associate with linear and nonlinear terms, respectively, that may be complex
or real, $G\in \mathbb{C}^{n}$ is continuous vector
function, $H\in\mathbb{C}^{n\times n}$. Note that if $B=0$, we get the formula \eqref{matrix:sys}.
Thus, our aim in this work is to design a general scheme to achieve adaptive synchronization with entirely unknown parameters applicable for systems \eqref{matrix:general:sys}.
\par Chaotic (hyperchaotic) systems play a significant role in secure communication systems due to their complicated behaviors and sensitive dependence on initial conditions. Numerous schemes of chaotic (hyperchaotic)-based secure communication systems are proposed in the literature. The ideas of their  techniques are follows: some schemes are designed to send the information signal between the transmitter and the receiver through one public channel, and the message by suitable encryption function is injected into the states of the transmitter (see e.g., \cite{Liao-Tsai-2000,Li-Xu-2004,Wu-Wang-2011,Mahmoud-Mahmoud-Arafa-2016}); other ones use the same algorithm, but instead of one public channel there are two channels: one to send the states of the transmitter system, and the second one to transmit the information-bearing signal. To get rapid synchronization and more security (see e.g., \cite{Shoreh-Kuznetsov-2020,Shoreh-Kuznetsov-2021,Jiang-2002}). Further, rather than inserting the message into the states of the transmitter, another type of scheme is designed, where the message is injected into the parameters of the transmitter, and by the adaptive laws of parameters the message can be recovered (see e.g., \cite{He-Cai-2016,Mahmoud-Farghaly-2018}).
\par In this article, we propose a novel hyperchaotic-based secure communication scheme. The scheme is designed to split the information signal and distribute it between two channels, which increases the security of the communication system and complicates the decryption task by intruders. Some bit of information signal is injected into parameters modulation and transmitted through one of the two channels; meanwhile,  the other bit is injected into the transmitter states and sent over the second channel. At the receiver side, the information signal can be accurately retrieved by adaptive techniques and a decryption function. The proposed scheme is robust for diverse scales of white Gaussian noise, which will be demonstrated below.
\par Another feature, which plays an important role in the quality of secure communication, is related to the questions of how nontrivial is the dynamics possessed by the system and how difficult it is to localize attractors in it (see e.g., \cite{Mahmoud-Farghalys-2017}).
To localize an attractor numerically, one needs to examine its basin of attraction and select an initial point in it. If the basin of attraction for a specific attractor is connected with the unstable manifold of an unstable equilibrium, then the localization procedure is rather simple.
From this perspective, the following classification of attractors is suggested \cite{Leonov-Kuznetsov-2011,Leonov-Kuznetsov-2013,Leonov-Kuznetsov-2015,Kuznetsov-Mokaev-2020}: an attractor is called a hidden attractor if its basin of attraction doesn't intersect with any small open neighborhood of an equilibrium point; otherwise, it is called a self-excited attractor.
For instance, hidden attractors are periodic or chaotic attractors in systems without equilibria, either with only one stable equilibrium, or with the coexistence of attractors in multistable engineering systems.
In general, localization of hidden attractors could be a challenging task, which requires developing of special analytical-numerical procedures \cite{Leonov-Kuznetsov-2013}.
\par In this article, a new scheme to achieve adaptive synchronization for a class of hyperchaotic complex-valued systems
with unknown parameters is proposed. Based on this kind of synchronization, a novel scheme for secure communication with improved a new cryptosystem is organised and tested on the basis of the complex-valued Rabinovich system.
\section{Adaptive synchronization for a general class of hyperchaotic complex-valued systems with unknown parameters}
In this section, we design a scheme to derive the control functions to achieve adaptive synchronization between two
identical complex-valued systems with entirely unknown parameters. Suppose the drive and response systems in the matrix form read:
\begin{equation}\label{adaptive:drive:sys}
  \dot{\textsf{x}}=F(\textsf{x})A+G(\textsf{x})+H(\textsf{x})B,
\end{equation}
\begin{equation}\label{adaptive:response:sys}
 \dot{\textsf{y}}=F(\textsf{y})\hat{A}+G(\textsf{y})+H(\textsf{y})\hat{B}+\theta(\textsf{x},\textsf{y}),
\end{equation}
where $\textsf{x},\textsf{y}\in \mathbb{C}^{n}$ are the complex state vectors, $F\in\mathbb{C}^{n\times n}$, $A,B$ are constant vectors
 of system parameters, $G\in \mathbb{C}^{n}$ is a continuous vector
function, $H\in\mathbb{C}^{n\times n}$, $\theta: \mathbb{C}^{n}\times \mathbb{C}^{n}\rightarrow \mathbb{C}^{n}$ is control
vector function of the response system \eqref{adaptive:response:sys}, the parameters modulation $\hat{A}, \hat{B}$ are defined as follows: $e_{A}=\hat{A}-A, e_{B}=\hat{B}-B$. The error can be written as
\begin{equation}\label{adaptive:error:sys}
  e(t)=\textsf{y}(t)-\textsf{x}(t).
\end{equation}
\theoremstyle{definition}
\begin{definition}
The derive system \eqref{adaptive:drive:sys} and response system \eqref{adaptive:response:sys} achieve adaptive synchronization if
\begin{equation*}
 \lim_{t\rightarrow \infty}||e(t)||=\lim_{t\rightarrow \infty}\parallel \textsf{y}(t)-\textsf{x}(t)\parallel=0.
\end{equation*}
\end{definition}
\par We will derive an appropriate
control to achieve adaptive synchronization between systems \eqref{adaptive:drive:sys} and \eqref{adaptive:response:sys}.
\begin{theorem}\label{theorem1:adaptive}
Adaptive synchronization between the derive system \eqref{adaptive:drive:sys} and the response system \eqref{adaptive:response:sys} will achieve if the control vector function is designed as follows:
\begin{multline}\label{adaptive:control:sys}
  \theta(\textsf{x},\textsf{y})=[F(\textsf{x})-F(\textsf{y})]A+ G(\textsf{x})-G(\textsf{y})+[H(\textsf{x})-H(\textsf{y})]B \\
  -K_{1} e-K_{2} e_{A}-K_{3} e_{B},
\end{multline}
where $K_{1}=diag(k_{11},k_{12},...,k_{1n})$, $K_{2}=diag(k_{21},k_{22},...,k_{2n})$, $K_{3}=diag(k_{31},k_{32},...,k_{3n})$ are positive definite diagonal gain matrices, and the updating of the parameters modulation are chosen as
\begin{equation}\label{param:modulation:A}
  \dot{e}_{A}=\dot{\hat{A}}=-(F(\textsf{y}))^{\ast}e+K_{2}e,
\end{equation}
\begin{equation}\label{param:modulation:B}
  \dot{e}_{B}=\dot{\hat{B}}=-(H(\textsf{y}))^{\ast}e+K_{3}e.
\end{equation}
\end{theorem}
\begin{proof}
The derivative of the synchronization error \eqref{adaptive:error:sys} to the time reads:
\begin{equation}\label{adaptive:error:derivat:sys}
  \dot{e}(t)=\dot{\textsf{y}}(t)-\dot{\textsf{x}}(t).
\end{equation}
From \eqref{adaptive:drive:sys} and \eqref{adaptive:response:sys}, we obtain
\begin{multline}\label{adaptive:error:derivat1:sys}
  \dot{e}(t)=[F(\textsf{y})-F(\textsf{x})]A+F(\textsf{y})e_{A}+G(\textsf{y})-G(\textsf{x}) \\
  +[H(\textsf{y})-H(\textsf{x})]B+H(\textsf{y})e_{B}+\theta.
\end{multline}
Define the following Lyapunov function:
\begin{equation}\label{adaptive:Lyapunov}
  V(\chi)=\frac{1}{2}\Big(e^{\ast}e+e_{A}^{\ast}e_{A}+e_{B}^{\ast}e_{B}\Big),
\end{equation}
where $\chi=(e,e_{A},e_{B})$.\\
The time derivative of $V$ is
\begin{equation}\label{adaptive:Lyapunov:derivat}
  \dot{V}(\chi)=\frac{1}{2}\Big(e^{\ast}\dot{e}+\dot{e}^{\ast}e+e_{A}^{\ast}\dot{e}_{A}+\dot{e}^{\ast}_{A}e_{A}
  +e_{B}^{\ast}\dot{e}_{B}+\dot{e}^{\ast}_{B}e_{B}\Big).
\end{equation}
Substituting \eqref{adaptive:control:sys}, \eqref{param:modulation:A}, \eqref{param:modulation:B} and \eqref{adaptive:error:derivat1:sys} in \eqref{adaptive:Lyapunov:derivat}, we get
\begin{multline}\label{adaptive:Lyapunov:derivat1}
  \dot{V}(\chi)=\frac{1}{2}\Big(e^{\ast}(F(\textsf{y})-F(\textsf{x}))A+e^{\ast}F(\textsf{y})e_{A}+e^{\ast}(G(\textsf{y})-G(\textsf{x}))\\
  +e^{\ast}(H(\textsf{y})-H(\textsf{x}))B+e^{\ast}H(\textsf{y})e_{B}+e^{\ast}(F(\textsf{x})-F(\textsf{y})A \\
  +e^{\ast}(G(\textsf{x})-G(\textsf{y}))+e^{\ast}(H(\textsf{x})-H(\textsf{y}))B-e^{\ast}K_{1} e-e^{\ast}K_{2} e_{A} \\
  -e^{\ast}K_{3} e_{B}+A^{\ast}(F(\textsf{y})-F(\textsf{x}))^{\ast}e+e_{A}^{\ast}(F(\textsf{y}))^{\ast}e+(G(\textsf{y})-G(\textsf{x}))^{\ast}e \\
  +B^{\ast}(H(\textsf{y})-H(\textsf{x}))^{\ast}e+e_{B}^{\ast}(H(\textsf{y}))^{\ast}e+A^{\ast}(F(\textsf{x})-F(\textsf{y}))^{\ast}e\\
  +(G(\textsf{x})-G(\textsf{y}))^{\ast}e+B^{\ast}(H(\textsf{x})-H(\textsf{y}))^{\ast}e-e^{\ast}K_{1} e-e_{A}^{\ast}K_{2}e \\
  -e_{B}^{\ast}K_{3}e -e_{A}^{\ast}(F(\textsf{y}))^{\ast}e+e_{A}^{\ast}K_{2}e-e^{\ast}(F(\textsf{y}))e_{A}+e^{\ast}K_{2}e_{A}\\
  -e_{B}^{\ast}(H(\textsf{y}))^{\ast}e+e_{B}^{\ast}K_{3}e-e^{\ast}(H(\textsf{y}))e_{B}+e^{\ast}K_{3}e_{B}\Big),
\end{multline}
then we have
\begin{equation}\label{adaptive:Lyapunov:derivat2}
  \dot{V}(\chi)=-e^{\ast}K_{1}e.
\end{equation}
Since $K_{1}$ is a positive definite matrix then $e^{\ast}K_{1}e>0.$ Hence, $\dot{V}<0$. It is obvious that $V(\chi)$ is positive definite, $\dot{V}(\chi)$ is negative definite and $V(\chi)\rightarrow \infty$ as $\chi\rightarrow \infty$. According to the Lyapunov stability theory, the error $\chi=(e,e_{A},e_{B})$ is asymptotically sable i.e., $\displaystyle$, thus $(e,e_{A},e_{B})\rightarrow0$ as $t\rightarrow\infty$. This completes the proof.
\end{proof}
\section{The hyperchaotic complex-valued Rabinovich system and its dynamics}
The complex-valued Rabinovich system describing the parametric excitation of waves in a magneto-active plasma can be written as follows \cite{Rabinovich-1978}:
\begin{equation}\label{Rabinovich:sys}
\left\{
  \begin{array}{ll}
    \dot{x}=-\upsilon x +y x^{\ast}+z y^{\ast},  \\
     \dot{y}=\alpha y -x^{2} +2 x^{\ast} z-\beta |y|^{2}y,  \\
     \dot{z}=\gamma z- 3 x y -\beta |z|^{2} z,
  \end{array}
\right.
\end{equation}
where $x=x_{1}+i x_{2}$, $y=x_{3}+i x_{4}$, $z= x_{5}+i  x_{6}$, $i=\sqrt{-1}$, a ''star'' superscript denotes the complex conjugate of a variable and $\upsilon, \alpha, \beta, \gamma$ are real parameters.
By equating real and imaginary parts of \eqref{Rabinovich:sys}, we obtain the following real-valued autonomous system:
\begin{equation}\label{Rabinovich:real:sys}
\left\{
  \begin{array}{ll}
    \dot{x}_{1}=-\upsilon x_{1}+ x_{1} x_{3}+ x_{2} x_{4}+ x_{3} x_{5}+ x_{4} x_{6},  \\
    \dot{x}_{2}=-\upsilon x_{2}+ x_{1} x_{4}- x_{2} x_{3}+ x_{3} x_{6}- x_{4} x_{5},  \\
    \dot{x}_{3}=\alpha x_{3}+ (x^{2}_{2} -x^{2}_{1})-2( x_{1} x_{5}+x_{2} x_{6})-\beta x_{3} (x^{2}_{3} + x^{2}_{4}),  \\
    \dot{x}_{4}=\alpha x_{4}- 2 x_{1} x_{2} + 2 ( x_{1} x_{6}- x_{2} x_{5})-\beta x_{4} (x^{2}_{3} + x^{2}_{4}),  \\
    \dot{x}_{5}=\gamma x_{5}-3 ( x_{1} x_{3}- x_{2} x_{4})-\beta x_{5} (x^{2}_{5} + x^{2}_{6}),  \\
     \dot{x}_{6}=\gamma x_{6}-3 ( x_{1} x_{4}+ x_{2} x_{3})-\beta x_{6} (x^{2}_{5} + x^{2}_{6}).  \\
  \end{array}
\right.
\end{equation}
Equating all equations in \eqref{Rabinovich:real:sys} with zero, we get $S_0=(0,0,0,0,0,0)$ is the only equilibrium point. The characteristic polynomial of the Jacobian matrix at $S_0$ is:
\begin{equation}\label{charac:equation}
  (\upsilon+\lambda)^{2}(\alpha-\lambda)^{2}(\gamma-\lambda)^{2}=0.
\end{equation}
For $\upsilon>0$, $\alpha<0$, $\gamma<0$, all eigenvalues of the characteristic polynomial \eqref{charac:equation} are negative; thus, the equilibrium $S_{0}$ is stable and otherwise is unstable.
\par For $\upsilon=-0.03$, $\alpha=0.5$, $\beta=0.001$  $\gamma=0.11$, $S_{0}$ is unstable and system \eqref{Rabinovich:real:sys} exhibits a hyperchaotic self-exited attractor as shown in Fig~\ref{fig:adaptive-self-excited}.
Using the following adaptive algorithm \cite{Kuznetsov-Leonov-2018}, it is possible to estimate the corresponding finite-time local Lyapunov exponents on the time interval $[0, 100]$ and initial point $(0.01,0.01,0.01,0.01,0.01,0.01)$ : $LE_{1} =  2.5124, LE_{2} = 0.1533, LE_{3} = 0.0571, LE_{4} = -0.0089, LE_{5} = -0.4718$, $LE_{6} = -2.5927$ (see Fig.~\ref{fig:adaptive-finit-Lyupunov}) and the finite-time local Lyapunov dimension is $LD=5.8648$.
\begin{figure}[!ht]
  \centering
  \includegraphics[width=\columnwidth]{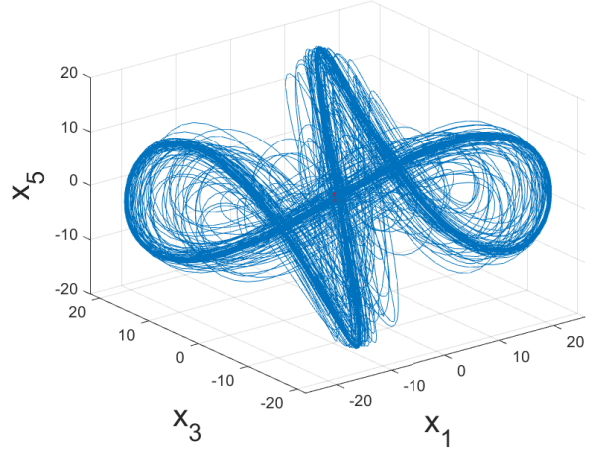}\\
  \caption{Visualization in projection $(x_1,x_3,x_5)$ of hyperchaotic  self-exited attractor of system \eqref{Rabinovich:real:sys} with $\upsilon=-0.03$, $\alpha=0.5$, $\beta=0.001$  $\gamma=0.11$ (self-excited with respect to
 the unstable zero equilibrium $S_0$).}
   \label{fig:adaptive-self-excited}
\end{figure}
\begin{figure}[!ht]
  \centering
  \includegraphics[width=\columnwidth]{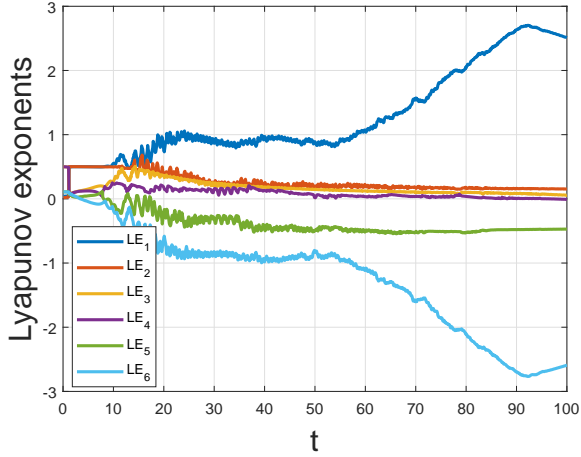}\\
  \caption{Finite-time local Lyapunov exponents on the time interval $[0, 100]$ of system \eqref{Rabinovich:real:sys} with $\upsilon=-0.03$, $\alpha=0.5$, $\beta=0.001$  $\gamma=0.11$.}
   \label{fig:adaptive-finit-Lyupunov}
\end{figure}

Since, currently, there are no results about dissipativity in the sense of Levinson and any criteria of global stability (see e.g. \cite{Leonov-Kuznetsov-2015}) for system \eqref{Rabinovich:sys}, the question of the existence of hidden attractors in system \eqref{Rabinovich:sys} is still open.
\par Further, we will use complex-valued Rabinovich system \eqref{Rabinovich:sys} for implementation of the new adaptive synchronization scheme with unknown parameters described above (see Eqs. \eqref{adaptive:drive:sys}-\eqref{adaptive:response:sys}, \eqref{adaptive:control:sys}). System \eqref{Rabinovich:sys} is not an artificial (man-made) one, it has a strong physical meaning and possesses a rich dynamics with possible chaotic or hyperchaotic attractors, which may be beneficial in the problems of synchronization and secure communication. Moreover, the structure of system \eqref{Rabinovich:sys} fits the design of the proposed synchronization scheme (see Eqs. \eqref{adaptive:drive:sys}-\eqref{adaptive:response:sys}) and is applicable to verify its efficiency. However, system \eqref{Rabinovich:sys} is considered here just as an example, and other complex-valued systems can be similarly studied.

\section{Adaptive synchronization of two identical hyperchaotic complex-valued Rabinovich systems with unknown parameters }
In this section, we apply our scheme to achieve adaptive synchronization between two identical hyperchaotic complex-valued Rabinovich systems with unknown parameters. The drive and the response systems can be defined as follows:
\begin{equation}\label{adaptive:Rabinovich:drive:sys}
\left\{
  \begin{array}{ll}
    \dot{x}_{d}=-\upsilon x_{d} +y_{d} x_{d}^{\ast}+z_{d} y_{d}^{\ast},  \\
     \dot{y}_{d}=\alpha y_{d} -x_{d}^{2} +2 x_{d}^{\ast} z_{d}-\beta |y_{d}|^{2}y_{d},  \\
     \dot{z}_{d}=\gamma z_{d}- 3 x_{d} y_{d} -\beta |z_{d}|^{2} z_{d},
  \end{array}
\right.
\end{equation}
and
\begin{equation}\label{adaptive:Rabinovich:resp:sys}
\left\{
  \begin{array}{ll}
    \dot{x}_{r}=-\hat{\upsilon} x_{r} +y_{r} x_{r}^{\ast}+z_{r} y_{r}^{\ast}+\theta_{1},  \\
     \dot{y}_{r}=\hat{\alpha} y_{r} -x_{r}^{2} +2 x_{r}^{\ast} z_{r}-\hat{\beta} |y_{r}|^{2}y_{r}+\theta_{2},  \\
     \dot{z}_{r}=\hat{\gamma} z_{r}- 3 x_{r} y_{r} -\hat{\beta} |z_{r}|^{2} z_{r}+\theta_{3},
  \end{array}
\right.
\end{equation}
where $\theta_{1}, \theta_{2}, \theta_{3}$ are complex control functions. Here we use
the subscript $d$ and $r$ for the drive and the response systems,
respectively.
\par Remark that the complex-valued Rabinovich systems \eqref{adaptive:Rabinovich:drive:sys} and \eqref{adaptive:Rabinovich:resp:sys} can't be written in the form \eqref{matrix:sys} to attain adaptive synchronization with entirely unknown parameters. However, it is possible to write the drive system \eqref{adaptive:Rabinovich:drive:sys} in the form \eqref{matrix:general:sys} such that: $\textsf{x}=(x_{d},y_{d},z_{d})^{T}=(x_{1}+i x_{2},x_{3}+i x_{4},x_{5}+i x_{6})^{T}$, $A=(\upsilon,\alpha,\gamma)^{T}$, $B=(0,\beta,\beta)^{T}$,
\begin{equation*}
F(\textsf{x})=\left(
            \begin{array}{ccc}
              -x_{d} & 0 & 0 \\
              0 & y_{d} & 0 \\
              0 & 0 & z_{d} \\
            \end{array}
          \right),
          G(\textsf{x})=\left(
       \begin{array}{c}
         y_{d} x_{d}^{\ast}+z_{d} y_{d}^{\ast} \\
         -x_{d}^{2} +2 x_{d}^{\ast} z_{d} \\
         - 3 x_{d} y_{d}  \\
       \end{array}
     \right),
\end{equation*}
\begin{equation*}
H(\textsf{x})=\left(
            \begin{array}{ccc}
              0 & 0 & 0 \\
              0 &-|y_{d}|^{2}y_{d} & 0 \\
              0 & 0 & -|z_{d}|^{2} z_{d} \\
            \end{array}
          \right).
\end{equation*}
In the same way, we can write the response system \eqref{adaptive:Rabinovich:resp:sys} in the form \eqref{adaptive:response:sys} such that: $\textsf{y}=(x_{r},y_{r},z_{r})^{T}=(y_{1}+i y_{2},y_{3}+i y_{4},y_{5}+i y_{6})^{T}$, $\hat{A}=(\hat{\upsilon},\hat{\alpha},\hat{\gamma})^{T}$, $B=(0,\hat{\beta},\hat{\beta})^{T}$, $\theta=(\theta_{2},\theta_{2},\theta_{3})^{T}$,
\begin{equation*}
F(\textsf{y})=\left(
            \begin{array}{ccc}
              -x_{r} & 0 & 0 \\
              0 & y_{r} & 0 \\
              0 & 0 & z_{r} \\
            \end{array}
          \right),
          G(\textsf{y})=\left(
       \begin{array}{c}
         y_{r} x_{r}^{\ast}+z_{r} y_{r}^{\ast} \\
         -x_{r}^{2} +2 x_{r}^{\ast} z_{r} \\
         - 3 x_{r} y_{r}  \\
       \end{array}
     \right),
\end{equation*}
\begin{equation*}
H(\textsf{y})=\left(
            \begin{array}{ccc}
              0 & 0 & 0 \\
              0 &-|y_{r}|^{2}y_{r} & 0 \\
              0 & 0 & -|z_{r}|^{2} z_{r} \\
            \end{array}
          \right).
\end{equation*}

Using Theorem \ref{theorem1:adaptive}, the control functions \eqref{adaptive:control:sys} can be written as follows:
\begin{multline}\label{adaptive:control:equ}
\theta(\textsf{x},\textsf{y})=\bigg(
 \upsilon(x_{r}-x_{d})+(y_{d} x_{d}^{\ast}+z_{d} y_{d}^{\ast}-y_{r} x_{r}^{\ast}-z_{r} y_{r}^{\ast})+\varphi_{1},
  \\
  \alpha(y_{d}-y_{r})-(x_{d}^{2} - 2 x_{d}^{\ast} z_{d}+\beta |y_{d}|^{2}y_{d}-x_{r}^{2} +2 x_{r}^{\ast} z_{r}\!-\!\beta |y_{r}|^{2}y_{r})+\varphi_{2}, \\
  \gamma(z_{d}-z_{r})-(3 x_{d} y_{d} +\beta |z_{d}|^{2} z_{d\tau}-3 x_{r} y_{r} -\beta |z_{r}|^{2} z_{r})+\varphi_{3}  \bigg)^{T},
\end{multline}
where $\varphi_{1}=-k_{11}(e_{1}+ie_{2})-k_{21}(e_{A_{1}}+ie_{A_{2}})-k_{31}(e_{B_{1}}+ie_{B_{2}})$, $\varphi_{2}=-k_{12}(e_{3}+ie_{4})-k_{22}(e_{A_{3}}+ie_{A_{4}})-k_{32}(e_{B_{3}}+ie_{B_{4}})$ and $\varphi_{3}=-k_{13}(e_{5}+ie_{6})-k_{23}(e_{A_{5}}+ie_{A_{6}})-k_{33}(e_{B_{5}}+ie_{B_{6}})$.
\par Consider the following values of system parameters: $\upsilon=-0.03$, $\alpha=0.5$, $\beta=0.001$  $\gamma=0.11$. For these parameters, system \eqref{adaptive:Rabinovich:drive:sys} exhibits self-excited hyperchaotic attractor (see Fig.~\ref{fig:adaptive-self-excited}). In order to establish the validity of the controller \eqref{adaptive:control:equ}, we choose the control gain matrices as follows: $K_{1}=diag(66,55,77)$, $K_{2}=diag(13,12,15)$, $K_{3}=diag(15,15,59)$. Derive system\eqref{adaptive:Rabinovich:drive:sys} and response system\eqref{adaptive:Rabinovich:resp:sys} with the controllers \eqref{adaptive:control:equ} are solved numerically. Fig.~\ref{fig:adaptive:sych} shows the solutions of \eqref{adaptive:Rabinovich:drive:sys} and \eqref{adaptive:Rabinovich:resp:sys} with the initial conditions $x(0)=(1+i,1+i,1+i)^{T}$ and $y(0)=(3+3i,3+3i,3+3i)^{T}$. In Fig.~\ref{fig:adaptive-error-e}, one sees that adaptive synchronization error $e(t)$ converges to zero and synchronization is achieved after a short time $t = T_{s} (T_{s} = 13s)$. From Figs~\ref{fig:adaptive-error-eA} and \ref{fig:adaptive-error-eB}, it is evident that the parameters modulation errors $e_{A}$ and $e_{B}$ approach zero. Thus, the estimation of unknown parameters $(\hat{\upsilon},\hat{\alpha},\hat{\beta},\hat{\gamma})$ converge to $(\upsilon,\alpha,\beta,\gamma)=(-0.03,0.5,0.001,0.11)$ (see Fig.~\ref{fig:parameters-modulation}).

\begin{figure*}[!ht]
 \centering
\subfloat[ {\scriptsize}
 ] {
 \label{fig:adaptive-synchx1x4x6}
 \includegraphics[width=0.45\textwidth]{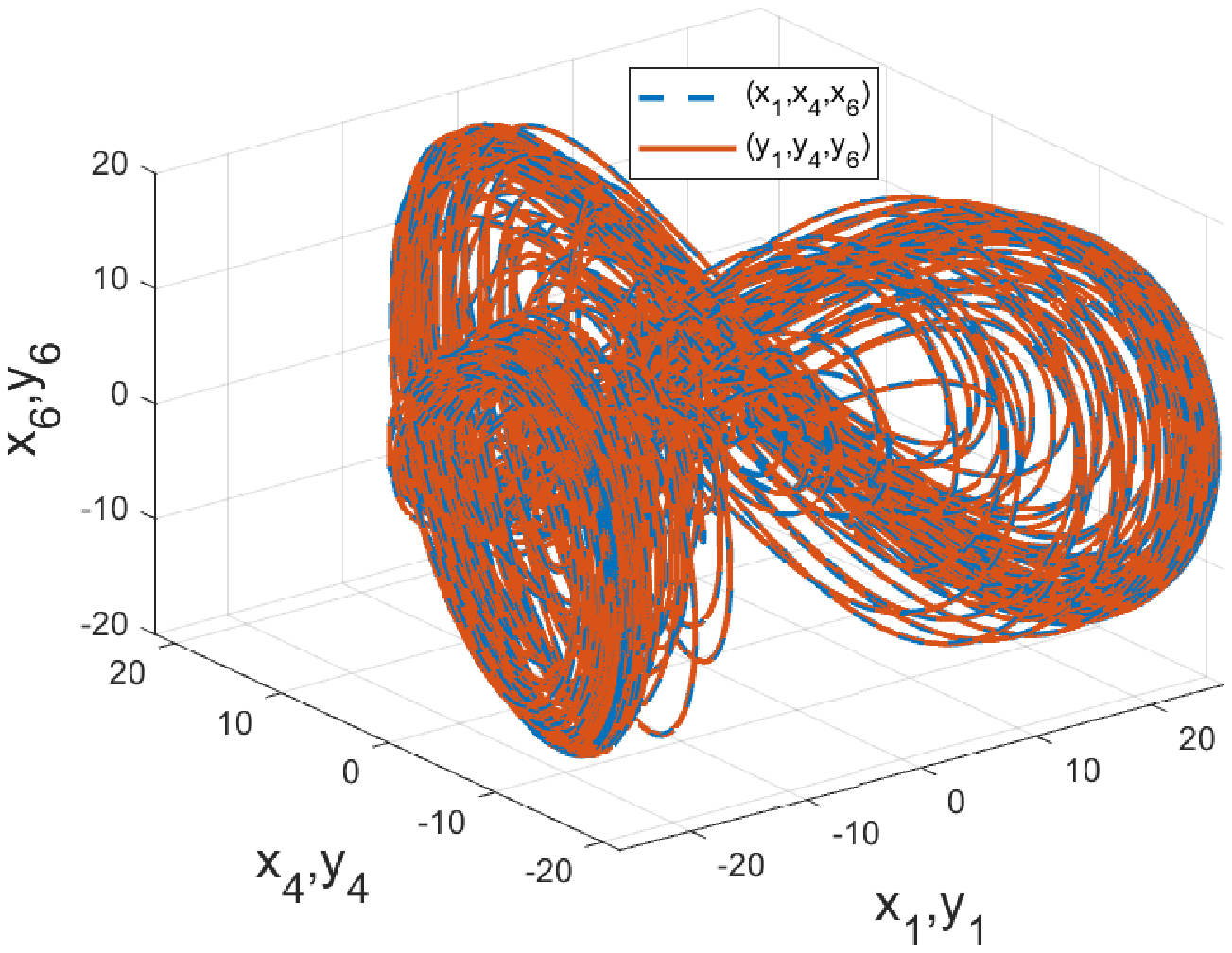}
 }~
 \subfloat[ {\scriptsize}
 ] {
 \label{fig:adaptive-synchx2x3x5}
 \includegraphics[width=0.45\textwidth]{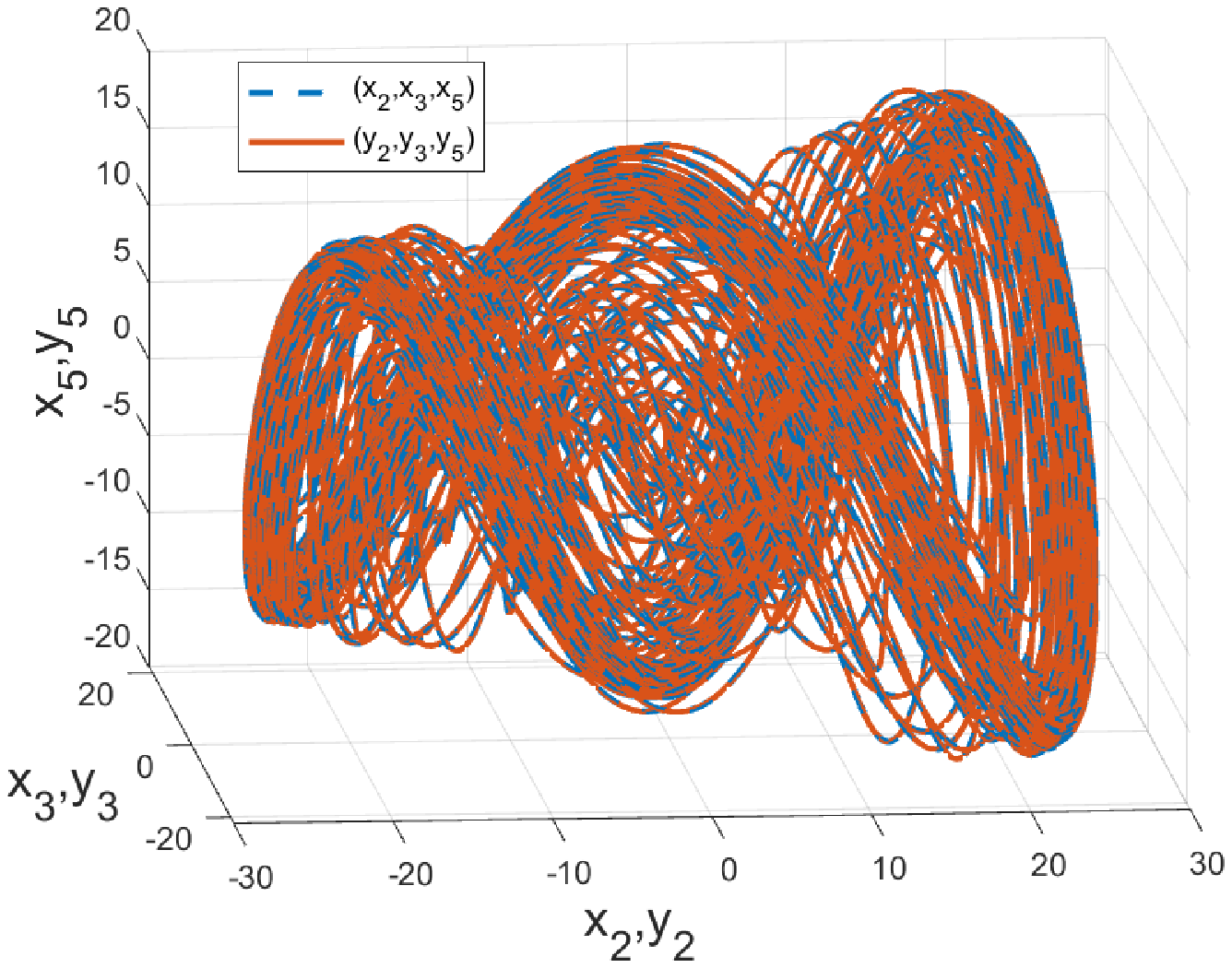}
 }

 \caption{Adaptive synchronization of self-excited hyperchaotic attractors of systems \eqref{adaptive:Rabinovich:drive:sys} and \eqref{adaptive:Rabinovich:resp:sys}, for $\upsilon=-0.03$, $\alpha=0.5$, $\beta=0.001$  $\gamma=0.11$.
 }
 \label{fig:adaptive:sych}
\end{figure*}
\begin{figure*}[!ht]
 \centering
 \subfloat[
 {\scriptsize The error between  the derive \eqref{adaptive:Rabinovich:drive:sys} and the response \eqref{adaptive:Rabinovich:resp:sys} systems described by the solutions of system \eqref{adaptive:error:derivat1:sys}.}
 ] {
 \label{fig:adaptive-error-e}
 \includegraphics[width=0.45\textwidth]{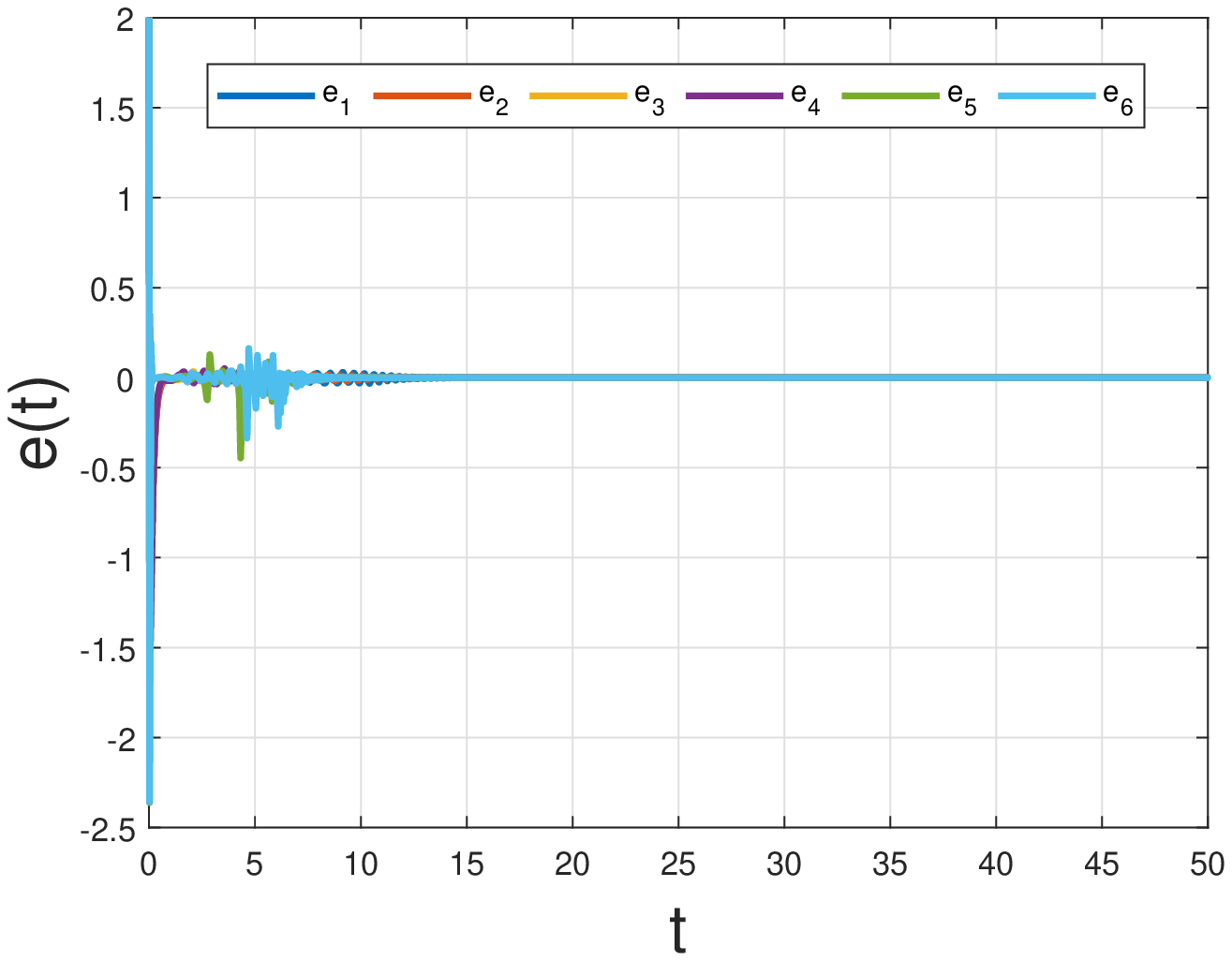}
 }~
 \subfloat[
 {\scriptsize The error of the parameters estimation $\hat{A}$ described by the solutions of system \eqref{param:modulation:A}.}
 ] {
 \label{fig:adaptive-error-eA}
 \includegraphics[width=0.45\textwidth]{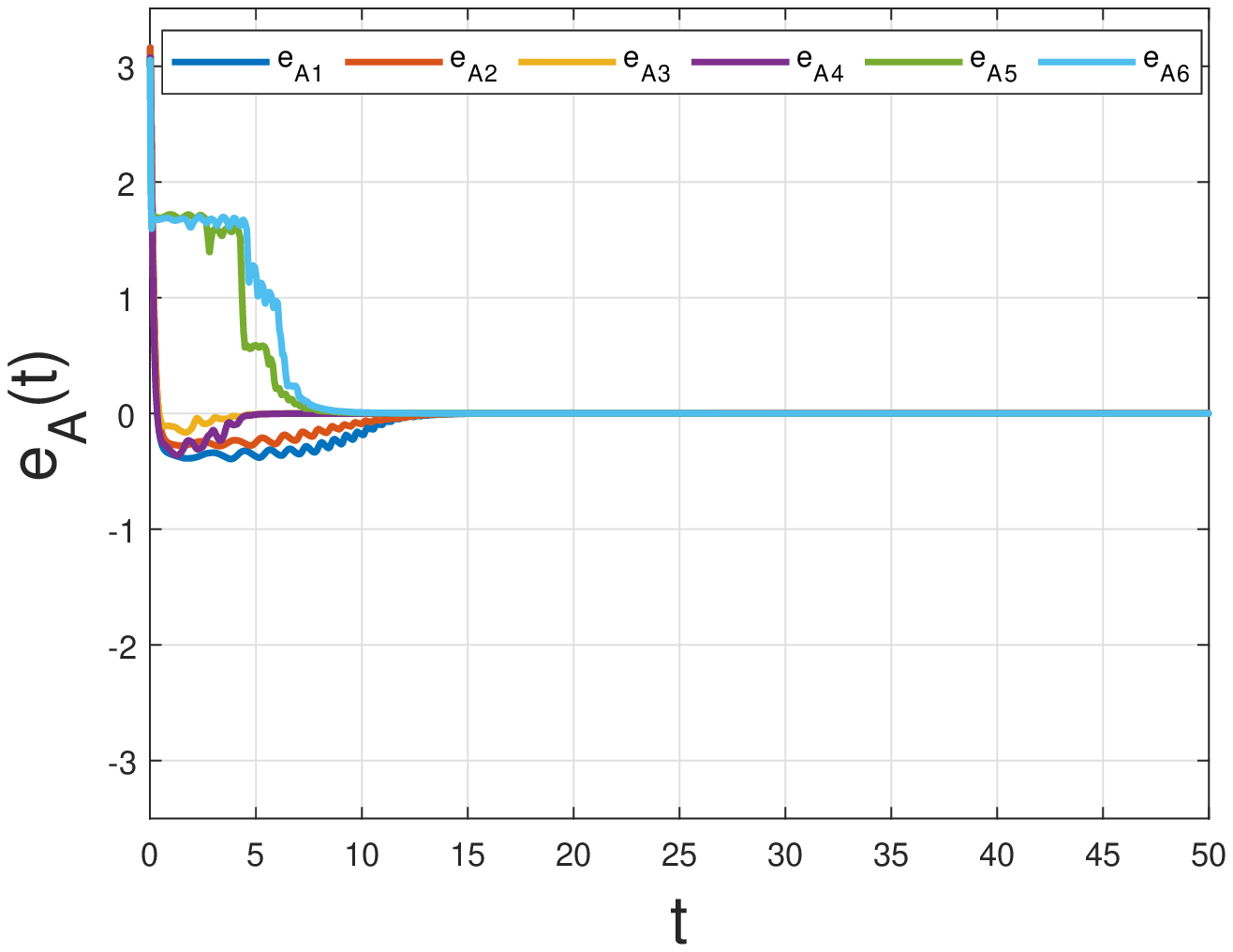}
 }
 \\
 \subfloat[
 {\scriptsize The error of the parameters estimation $\hat{B}$ described by the solutions of system \eqref{param:modulation:B}.}
 ] {
 \label{fig:adaptive-error-eB}
 \includegraphics[width=0.45\textwidth]{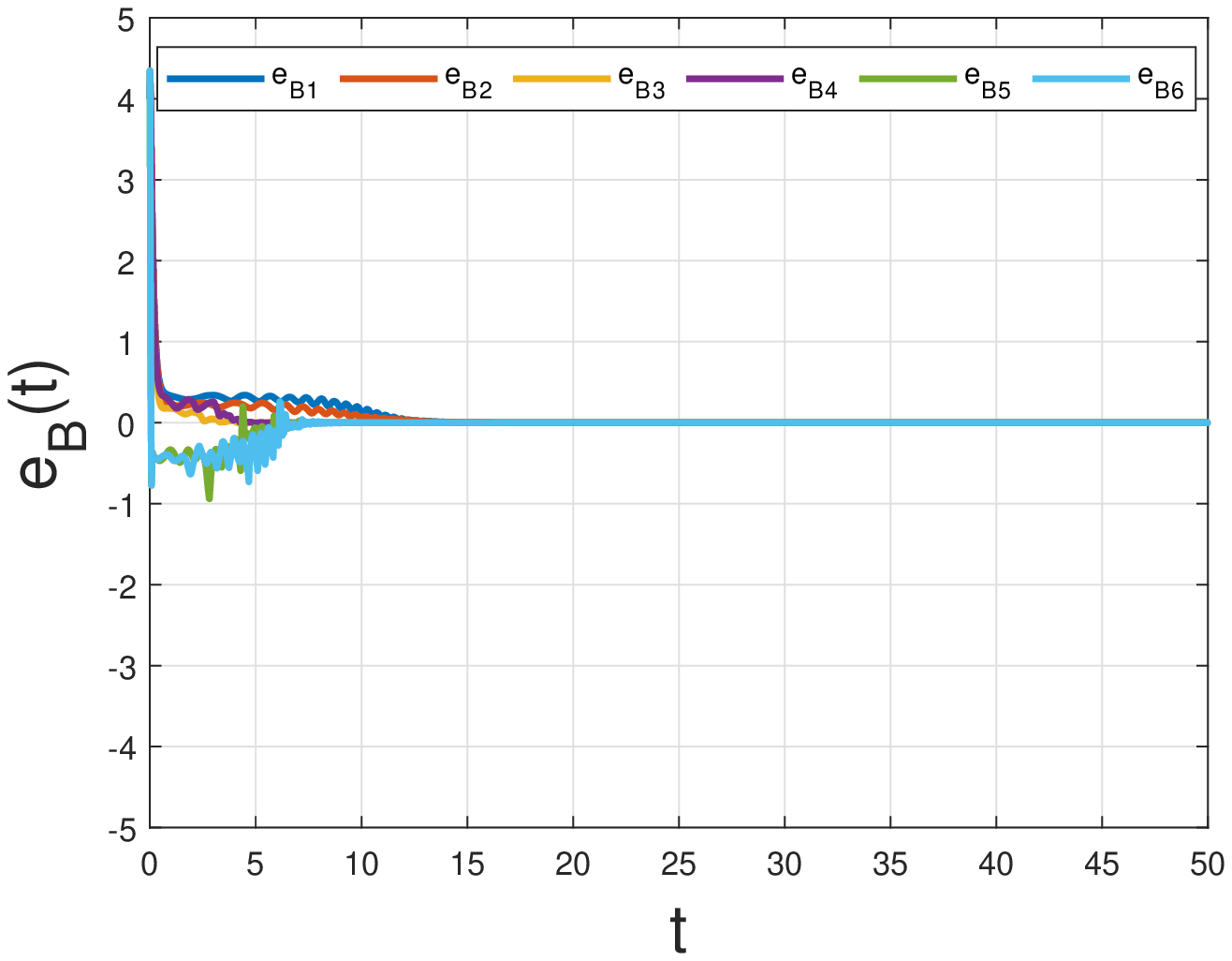}
 }
 \caption{Evolution of errors over the time.
 }
 \label{fig:adaptive:error}
\end{figure*}
\begin{figure*}[!ht]
  \centering
  \includegraphics[width=15cm]{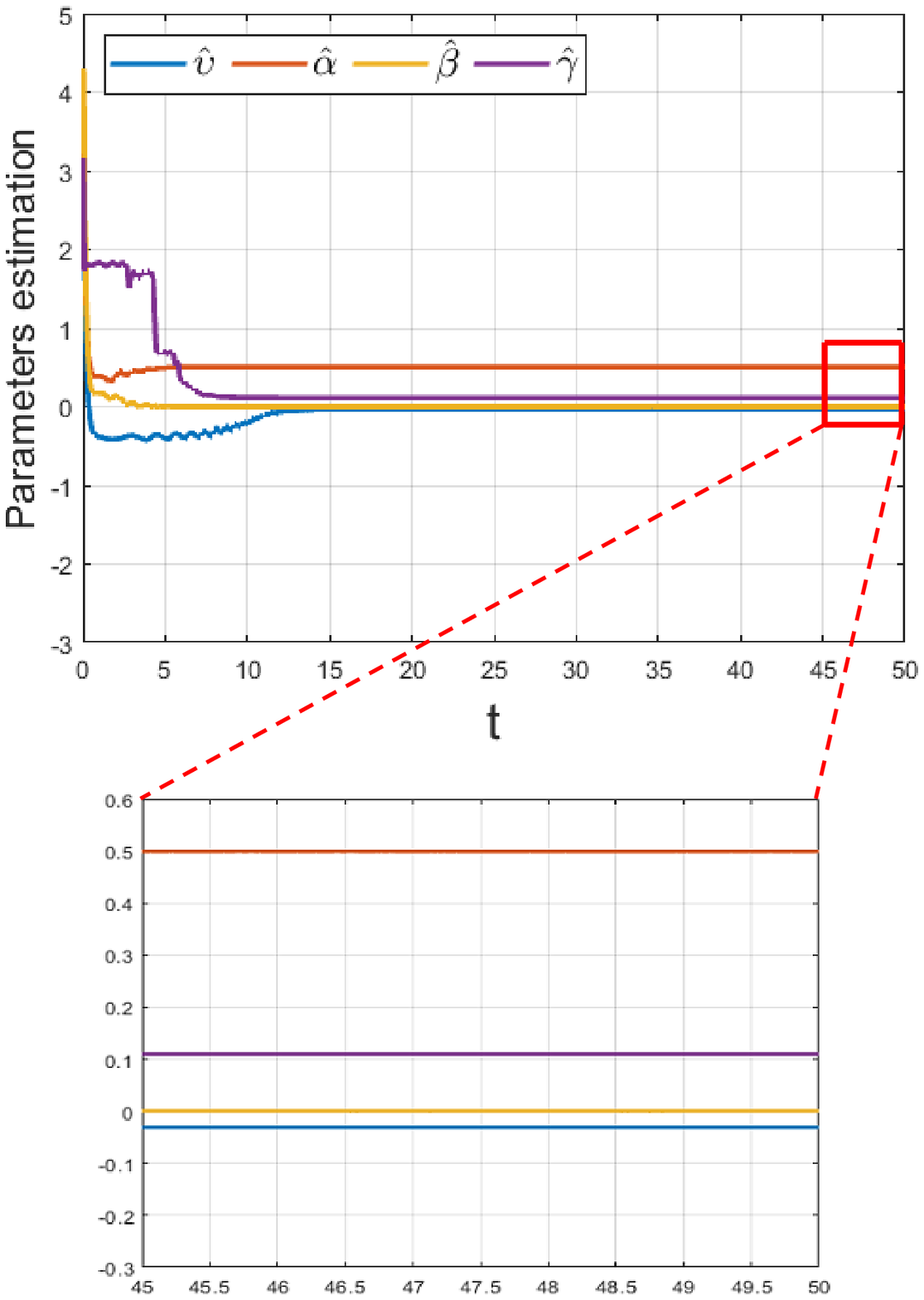}\\
  \vspace{-4.5cm}
  \caption{Estimation of unknown parameters $(\hat{\upsilon},\hat{\alpha},\hat{\beta},\hat{\gamma})$.}
   \label{fig:parameters-modulation}
\end{figure*}

\section{Application to secure communications}
In this section, we propose a novel scheme for secure communications based on the adaptive synchronization of hyperchaotic systems with uncertain parameters. In general, existing communication schemes are designed by injection of the message into chaotic states (chaotic masking) \cite{Liao-Tsai-2000,Li-Xu-2004,Wu-Wang-2011,Mahmoud-Mahmoud-Arafa-2016}, or parameters of the transmitter system (parameters modulation) \cite{He-Cai-2016,Mahmoud-Farghaly-2018} of the transmitter. The new scheme is built to split the message and inject some bit of information signal into parameter modulation and the other bit into the transmitter states, which improve the security of communication system and complicates decryption task by meddlers. The diagram of the scheme is shown in Fig.~\ref{fig:secure1}.
\begin{figure*}
  \centering
  \includegraphics[width=17cm]{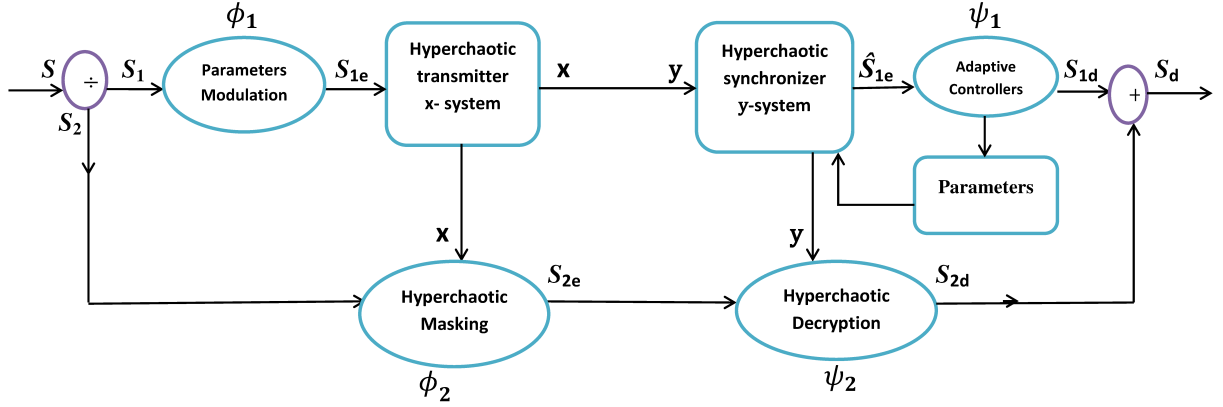}\\
 \vspace{-5.5cm} \caption{Hyperchaotic secure-communication scheme.}
   \label{fig:secure1}
\end{figure*}
The scheme consists of the following elements:
\begin{itemize}
  \item Hyperchaotic transmitter and receiver systems, which generate the pseudo-states variables $\textsf{x}(t)\in \mathbb{C}^{n}$ and $\textsf{y}(t)\in \mathbb{C}^{n}$, respectively.
  \item Splitting block: the message signal $S(t)$ is divided into two vectors of bits $S_1$ and $S_2$ to distribute it over two channels.
   \item Parameters modulation block: the first part of the message $S_1(t)$ by a continuous invertible function $\phi_1$ is modulated into the parameters of the transmitter.
  \item Hyperchaotic masking block: the second part of the message signal $S_2(t)$ to be encrypted is injected into a non-linear function $ \phi_2 : \mathbb{R}^{n} \times \mathbb{R}\longrightarrow \mathbb{R} $ that is continuous in its first argument $x \in \mathbb{R}^{n}$ and satisfies the following property: for all fixed pair of $(\textsf{x},S_2)\in \mathbb{R}^{n} \times \mathbb{R} $, there is a unique function $ \psi_2 : \mathbb{R}^{n} \times \mathbb{R}\longrightarrow \mathbb{R}$ that is continuous in its first argument $\textsf{x} \in \mathbb{R}^{n}$, such that $\psi_2(\textsf{x}, \phi_2(\textsf{x},S_2))=S_2$. The encryption function $\phi_2$ is built in terms of the hyperchaotic states. The result is a signal $S_{2e}(t)$ containing the message part that is sent over one of the channels.
  \item Channels: two channels transmit the hyperchaotic-state signals which contain the parameters modulation and the encrypted information-bearing signal $S_{2e}$.
  \item Synchronization block: for the receiver side of the communication system, a synchronization block is implemented to retrieve the hyperchaotic pseudo-state signals and provide the necessary information for the decryption.
  \item Adaptive controllers block: in the receiver side the adaptive controllers are built to trace the parameters of the transmitter system. After synchronization is realized the decryption function $\psi_1$ can be utilized to recover the first part of the transmitted message $S_{1d}$.
  \item Hyperchaotic decryption block: the masked message $S_{2e}(t)$ is decrypted by using a nonlinear
function $S_{2d}(t)=\psi_{2}(\textsf{y}(t),S_{1e}(t))$. In this case, $\textsf{y}(t)$ after a specific time $t=T_{s}$ is the estimation of the hyperchaotic state $\textsf{x}(t)$ generated by the synchronization block.
\item Gathering block: by combining the two information signals $S_{1d}$ and $S_{2d}$, we get the whole decrypted message $S_{d}$.
\end{itemize}
\par Unlike other schemes which could be inaccurate due to lack
of robustness and security (see e.g., \cite{Alvarez-Montoya-2004,JinFeng-JingBo-2008}), the proposed secure communication scheme is robust against the infiltration attempt by intruders because it must precisely identify the following information to decrypt the message successfully:\\
1. which part of the information bit is injected into the transmitter's parameters, and which part is inserted into hyperchaotic states of the transmitter;\\
2. which channel is used to transmit the transmitter's hyperchaotic signals with parameters modulation, and which one to send the hyperchaotic masked message part;\\
3. which parameters are utilized for parameters modulation of the transmitter;\\
4. on which basis of attractor, the hyperchaotic encryption was organized. For this task, the presence of multi-stability in the system and hidden attractors will significantly complicate decrypting;\\
5. the constructed functions for parameters modulation and hyperchaotic masking;\\
6. the adaptive laws of controllers (one of the most critical information to know).

\begin{remark}
The proposed secure communication scheme is built based on the technique of adaptive synchronization with unknown parameters, which is a tricky task for implementation and, in turn, leads to more security \cite{Mahmoud-2012,Xu-Zhou-2010,Liao-Lin-1999}.
\end{remark}
\begin{remark}
Using synchronization in communication systems is fundamental since it leads to recovering
basis function(s) in coherent correlation receivers (see e.g., \cite{Kolumban-Kennedy-1997,Kolumban-Kennedy-1998}).
\end{remark}
\par To illustrate secure communication scheme the complex-valued Rabinovich system is considered and the numerical results are executed in MATLAB. The transmitter is the implementation of \eqref{adaptive:drive:sys} which generates the states $(x_{1}+ix_{2},x_{3}+ix_{4},x_{5}+ix_{6})^{T}$ and the receiver is the implementation of \eqref{adaptive:response:sys} which generates the states $(y_{1}+iy_{2},y_{3}+iy_{4},y_{5}+iy_{6})^{T}$. As indicated in Fig.~\ref{fig:adaptive:sych}, derive system \eqref{adaptive:drive:sys} and response system \eqref{adaptive:response:sys} for $\upsilon=-0.03$, $\alpha=0.5$, $\beta=0.001$  $\gamma=0.11$ exhibit hyperchaotic self-excited attractors. The adaptive synchronization is achieved after $t=T_s (T_s=13s)$ if the control functions are chosen in the form \eqref{adaptive:control:equ} (see Fig.~\ref{fig:adaptive-error-e}). To demonstrate the robustness of our scheme, the message $S(t)$ is selected in two forms of plain text and gray image with diverse scales of white Gaussian noise.
\subsection{Single parameter modulation and hyperchaotic masking for a plain text encryption  }

The first type of the message $S(t)$ is a plain text which may contain alphabets, symbols, numbers, and spaces. To convert this text message into a vector of numbers, we use the package "double" in MATLAB software. Suppose $S=[s_{1},s_{2},..,s_{m},s_{m+1},s_{m+2},...,s_{n}]$, according to the proposed scheme, this vector of numbers is split into two vectors $S_{1}=[s_{1},s_{2},..,s_{m}]$ and $S_{2}=[s_{m+1},s_{m+2},...,s_{n}]$. The first vector $S_{1}$ by the following parameters modulation function is injected into the parameters of the transmitter:
\begin{equation}\label{encryptin:func1}
  S_{1e}=\phi_{1}(\{A,B\},S_{1})=\frac{S_{1}}{10\,d}+\gamma,
\end{equation}
where $d=\max(S(j))-\min(S(j))$, $j=1,2,...,n$, $\gamma=0.11$. For the second part of the
 message, $S_{2}$ is injected into the hyperchaotic states of the transmitter system using the following function:
 \begin{equation}\label{encryptin:func2}
  S_{2e}=\phi_{2}(\textsf{x},S_{2})=x_{6}^{2}+(1+x^{2}_{6})S_{2}.
\end{equation}\\
The receiver can decode the message by using the following decryption functions:
\begin{equation}\label{decryptin:func1}
  S_{1d}=\psi_{1}(\{\hat{A},\hat{B}\},\hat{S}_{1e})=10(\hat{S}_{1e}-\hat{\gamma})d,
\end{equation}
where $\hat{\gamma}=0.11$.
\begin{equation}\label{decryptin:func2}
  S_{2d}=\psi_{2}(\textsf{y},S_{2e})=-\frac{y_{6}^{2}}{1+y^{2}_{6}}+\frac{S_{2e}}{1+y^{2}_{6}}.
\end{equation}
By gathering the vectors $S_{1d}$ and $S_{2d}$, we get the whole decrypted information signal $S_{d}$. Using the package "char" in MATLAB software, the receiver can convert these numbers back to a plain text. For example, suppose we want to send a part of this article, e.g., the abstract, as shown in Fig.~\ref{fig:adaptive_original1}. The encrypted message is depicted in Fig.~\ref{fig:adaptive_encrypted}, and one can check that the transmitted message is strongly coded. Fig.~\ref{fig:adaptive_original2} shows the decrypted message, which matches the original text.
\begin{figure}[!ht]
  \centering
  \includegraphics[width=\columnwidth]{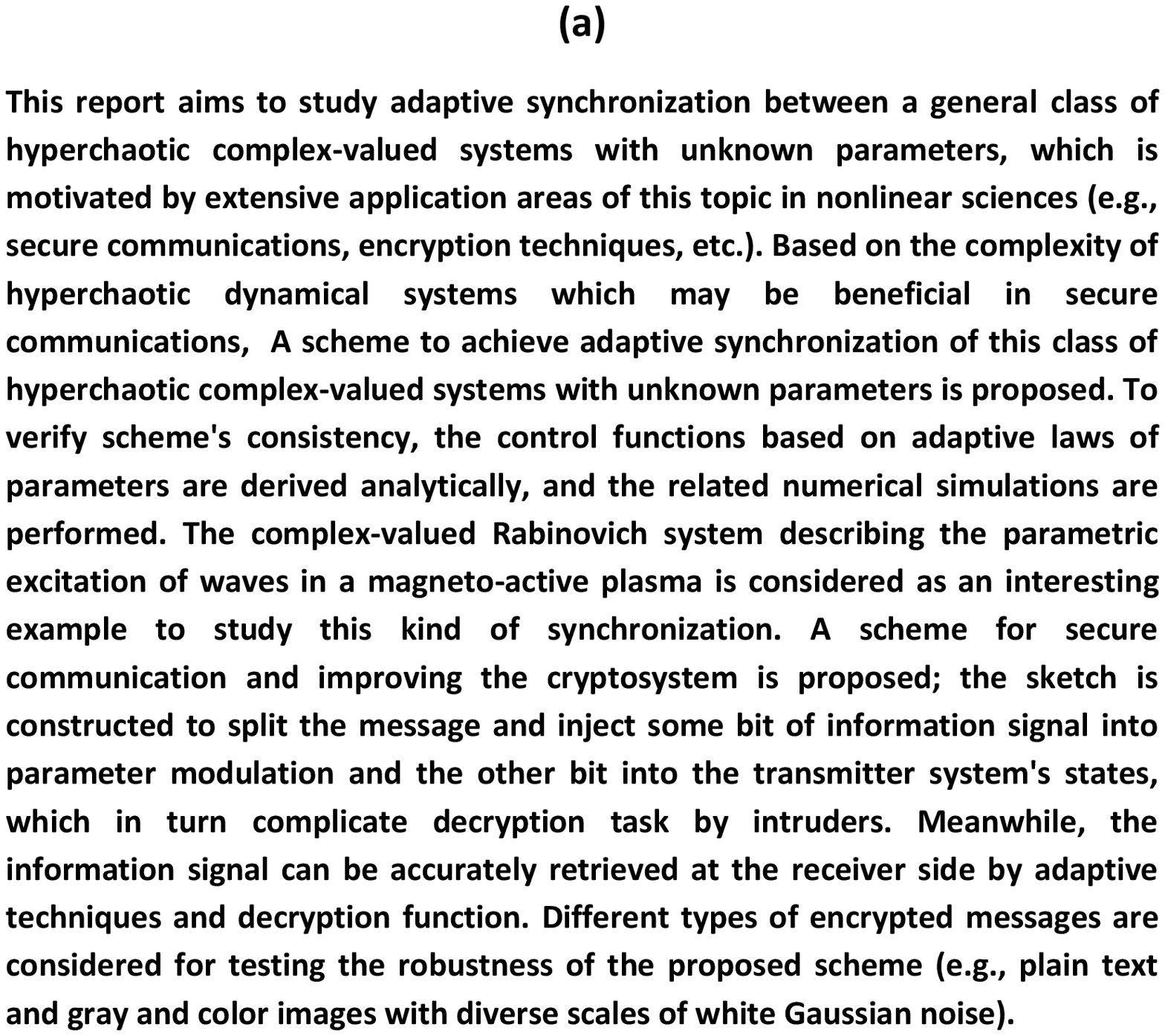}\\
  \vspace{-4cm}\caption{Original text.}
   \label{fig:adaptive_original1}
\end{figure}
\begin{figure}[!ht]
  \centering
  \includegraphics[width=\columnwidth]{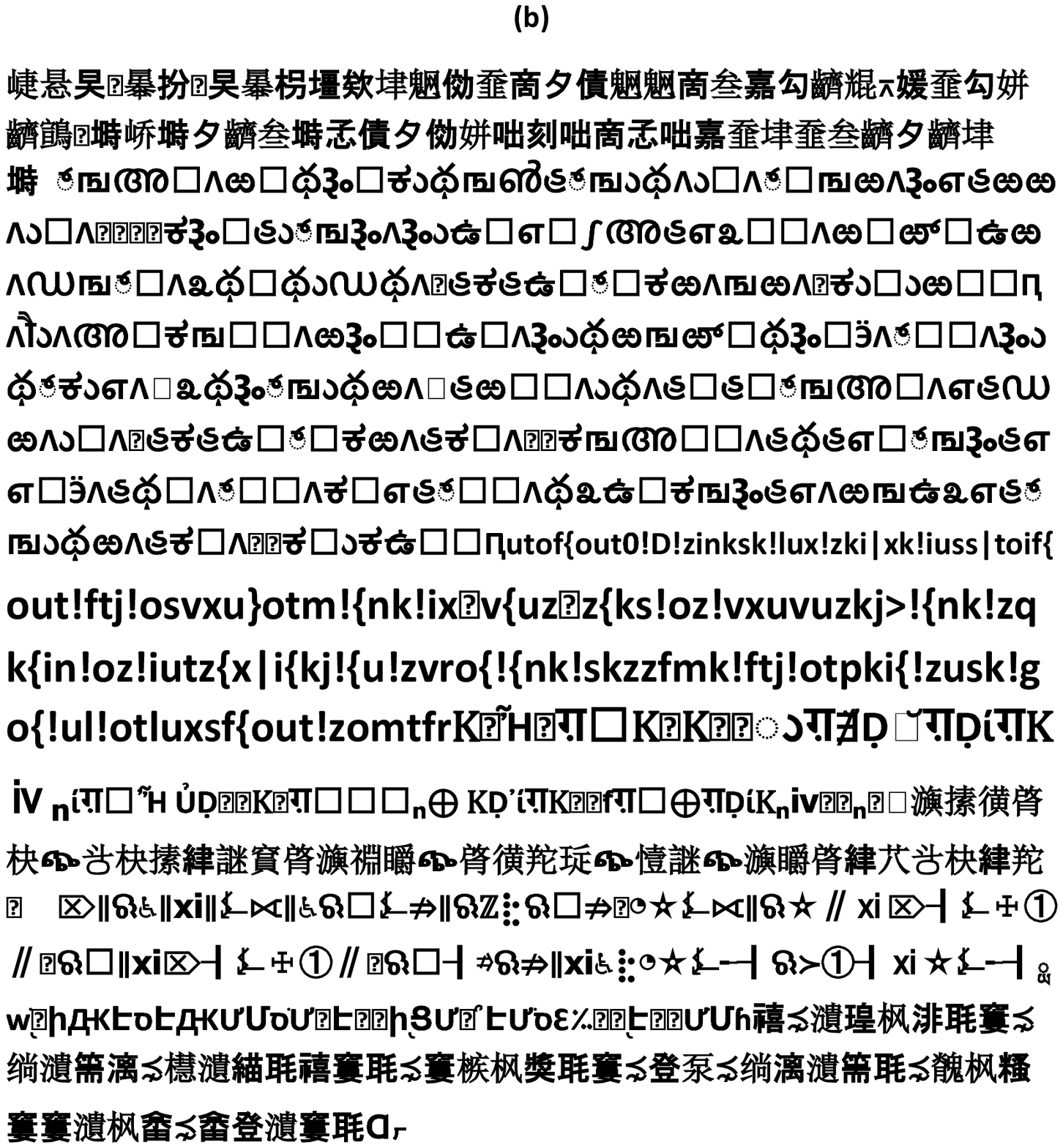}\\
 \vspace{-3cm} \caption{Encrypted text.}
   \label{fig:adaptive_encrypted}
\end{figure}

\begin{figure}[!ht]
  \centering
  \includegraphics[width=\columnwidth]{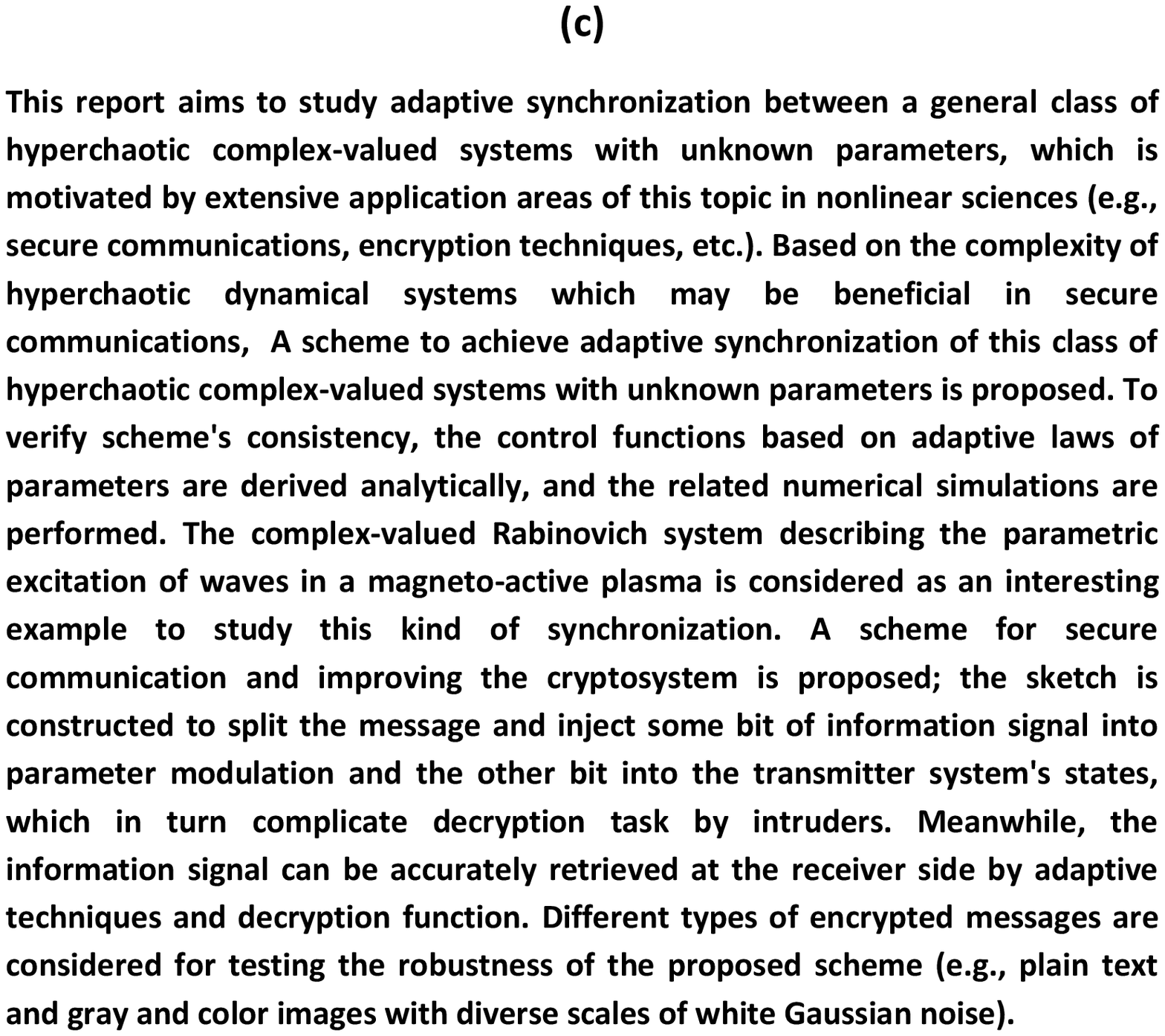}\\
  \vspace{-4cm}\caption{Decrypted text.}
   \label{fig:adaptive_original2}
\end{figure}
\subsection{Single parameter modulation and hyperchaotic masking for a grayscale image  }
Now we consider the second message in the form of grayscale image (cameraman.tif) with size $256\times256$. This image can be transformed into
$m\times n$ matrix of pixels as follows:
\begin{equation}\label{image:matrix}
  S=\left(
      \begin{array}{cccc}
        s_{11} & s_{12} & \cdots & s_{1n} \\
        s_{21} & s_{22} & \cdots & s_{2n} \\
        \vdots & \vdots & \ddots & \vdots \\
        s_{m1} & s_{m2} & \cdots & s_{mn} \\
      \end{array}
    \right),
\end{equation}
where $s(k,l)$ stands for the image value in pixel at the position $(k,l)$ where $k=1,2,...,m$, $l=1,2,...,n$. The matrix of pixels is converted into a 1-dimensional vector of integers between 0 and 255. Let $S=[s_{11},s_{21},...,s_{m1},s_{12},...,s_{m2},s_{1n},...,s_{mn}]=[s_{1},s_{2},...,s_{mn}]$. The last vector is divided into two vectors $S_{1}=[s_{1},s_{2},...,s_{k}]$ and $S_{2}=[s_{k+1},s_{k+2},...,s_{mn}]$. The first vector is injected into the transmitter system's parameters, while the second is injected into its hyperchaotic states. We use the same functions for parameter modulation, hyperchaotic masking, and decryption as in the first message. The transmitted and recovered grayscale images are shown in Figs.~\ref{fig:original_Image} and \ref{fig:decrypted_Image}.
The image histogram, which shows the distribution of intensities for original and decrypted images, are depicted in Figs.~\ref{fig:orignal_histogram}, and \ref{fig:decrypted_histogram}, respectively.
Through Figs. \ref{fig:decrypted_Image} and \ref{fig:decrypted_histogram}, one can observe that the original image (cameraman.tif) is accurately retrieved.
\begin{figure*}[!ht]
 \centering
 \subfloat[
 {\scriptsize Original image.}
 ] {
 \label{fig:original_Image}
 \includegraphics[width=0.45\textwidth]{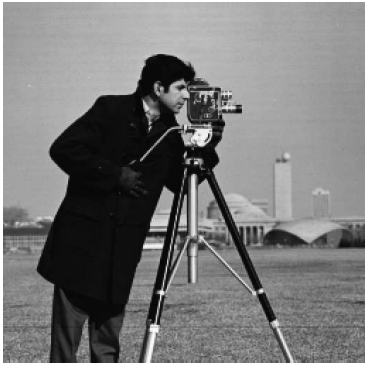}
 }~
 \subfloat[
 {\scriptsize Histogram of the original image.}
 ] {
 \label{fig:orignal_histogram}
 \includegraphics[width=0.45\textwidth]{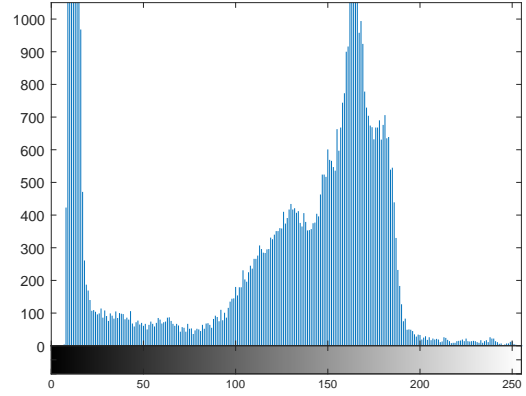}
 }
 \\
 \centering
 \subfloat[
 {\scriptsize Decrypted image.}
 ] {
 \label{fig:decrypted_Image}
 \includegraphics[width=0.45\textwidth]{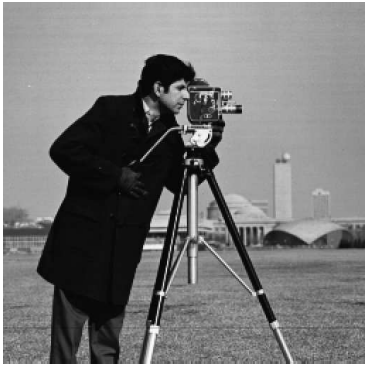}
 }~
 \subfloat[
 {\scriptsize Histogram of the decrypted image.}
 ] {
 \label{fig:decrypted_histogram}
 \includegraphics[width=0.45\textwidth]{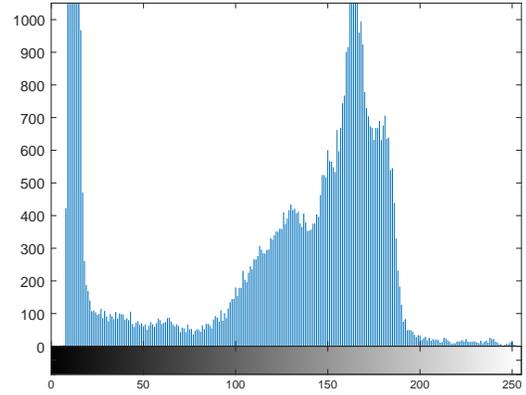}
 }
 \caption{Simulation results of grayscale image encryption using the complex-valued Rabinovich system.
 }
 \label{fig:adaptive:encrypted_decrypted}
\end{figure*}
\par To demonstrate the robustness of the secure communication scheme for image encryption, the white Gaussian noise with various scales is added to the grayscale image (cameraman.tif), and different tests are used to measure the quality of the retrieved image, e.g., peak signal-to-noise ratio (PSNR) and structural similarity image index (SSIM).
\subsubsection{Peak signal-to-noise ratio analysis}
To analyze the pixel distribution for the recovered image with respect to the original image, the peak signal-to-noise ratio (PSNR) is applied. The PSNR can be defined as the following \cite{Wang-Bovik-2004,Mahmoud-Farghaly-2018}:
\begin{equation}\label{PSNR}
P\!S\!N\!R(S,S_{d})=10\log_{10}\Bigg[\frac{(255)^{2}}{M\!S\!E(S,S_{d})}\Bigg],
\end{equation}
where MSE is the mean square error and is defined as follows:
\begin{equation}\label{MSE}
M\!S\!E(S,S_{d})=\frac{1}{m\times n}\sum_{k=1}^{m}\sum_{l=1}^{n}(s_{d}(k,l)-s(k,l))^{2},
\end{equation}
where $S$ and $S_{d}$ stand for the original and retrieved image, respectively.\\
Remark that a high value of PSNR refers to a close resemblance between the decrypted image and the original one.
\subsubsection{The structural similarity image index}
The second test considered to measure and analyze the similarity between the original and retrieved image is the structural similarity image index (SSIM) defined in the following form \cite{Wang-Bovik-2002}:
\begin{equation}\label{SSIM}
S\!S\!I\!M(S,S_{d})=\frac{(2\mu_{S}\mu_{S_{d}}+C_{1})(2\sigma_{SS_{d}}+C_{2})}{(\mu_{S}^{2}+\mu_{S_{d}}^{2}+C_{1})(\sigma_{S}^{2}+\sigma_{S_{d}}^{2}+C_{2})},
\end{equation}
where $\mu_{S}$ and $\mu_{S_{d}}$ are the average luminance value of original image $S$ and the decrypted image $S_{d}$, respectively;
$\sigma_{S}$ and $\sigma_{S_{d}}$ are the standard variances of $S$ and $S_{d}$, respectively; $\sigma_{SS_{d}}$ is the covariance between $S$ and $S_{d}$; $C_{1}$ and $C_{2}$ are small fixed positive constants in order to have the denominator not equal to zero. The estimate of the SSIM is always in the interval $[-1,1]$, and the strongest estimate 1 is realized if $S=S_{d}$.
In order to demonstrate the robustness of the proposed scheme to noise, the grayscale image (cameraman.tif) is transmitted from the transmitter to the receiver with various levels of white Gaussian noise. Fig.~\ref{fig:adaptive:noise} shows the original noise images and the retrieved noise ones. The estimate of the PSNR and SSIM for the retrieved grayscale images with the original ones are listed in Table 1.\\
Table 1. Estimation of PSNR and SSIM for grayscale cameraman.tif image.\\\\
\begin{tabular}{c| c| c}
  \hline
  Gray image & PSNR & SSIM \\
  \hline
  Cameraman (Gaussian noise 0.03) & 77.9705 & 0.9999 \\
   Cameraman (Gaussian noise 0.07) & 77.3193 & 0.9999 \\
   Cameraman (Gaussian noise 0.1) & 77.4875 & 0.9999 \\
  \hline
\end{tabular}
\begin{figure*}[!ht]
 \centering
 \subfloat[
 {\scriptsize }
 ] {
 \label{fig:noise3}
 \includegraphics[width=0.35\textwidth]{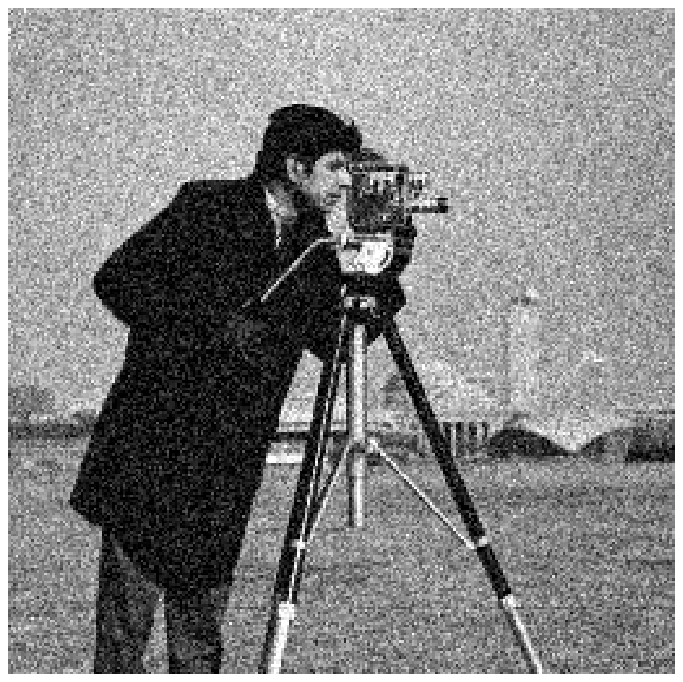}
 }~
 \subfloat[
 {\scriptsize }
 ] {
 \label{fig:noise7}
 \includegraphics[width=0.35\textwidth]{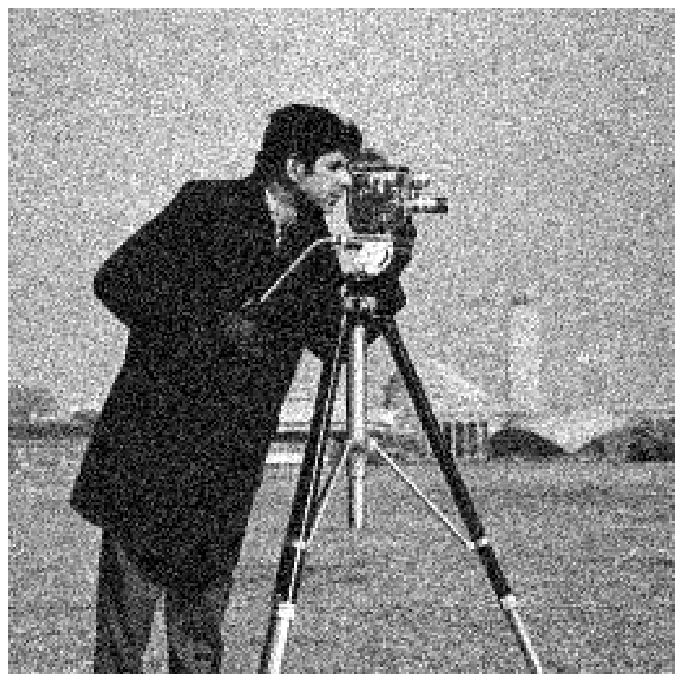}
 }
  \centering
 \subfloat[
 {\scriptsize }
 ] {
 \label{fig:noise1}
 \includegraphics[width=0.35\textwidth]{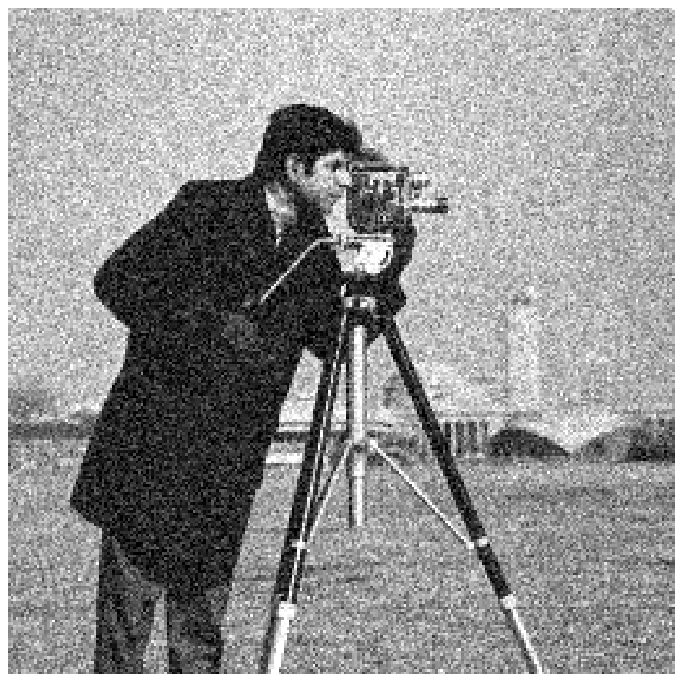}
 }~
 \\
 \subfloat[
 {\scriptsize }
 ] {
 \label{fig:noise3decrypt}
 \includegraphics[width=0.35\textwidth]{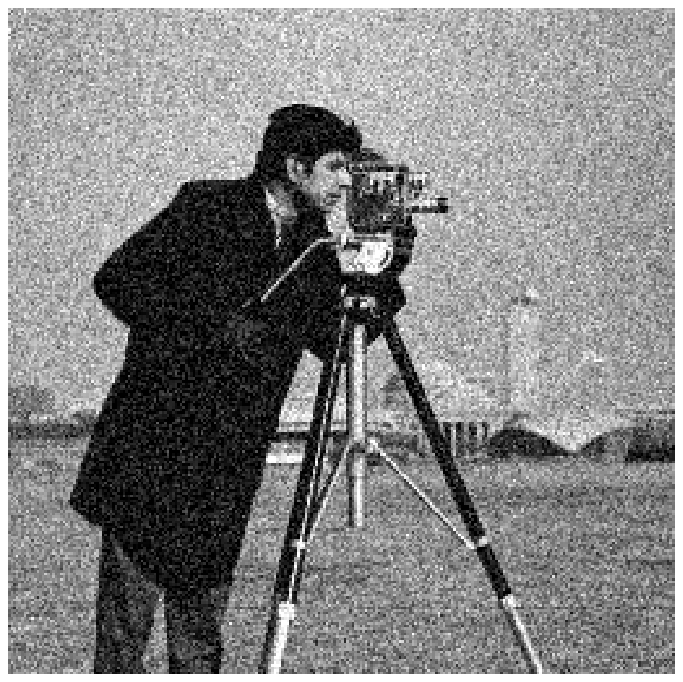}
 }
 \centering
 \subfloat[
 {\scriptsize }
 ] {
 \label{fig:noise7decrypt}
 \includegraphics[width=0.35\textwidth]{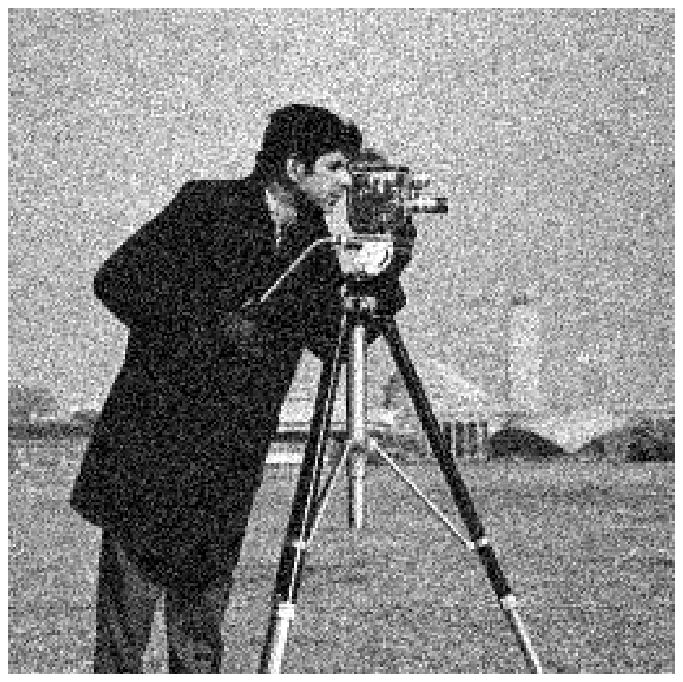}
 }~
 \subfloat[
 {\scriptsize }
 ] {
 \label{fig:noise1decrypt}
 \includegraphics[width=0.35\textwidth]{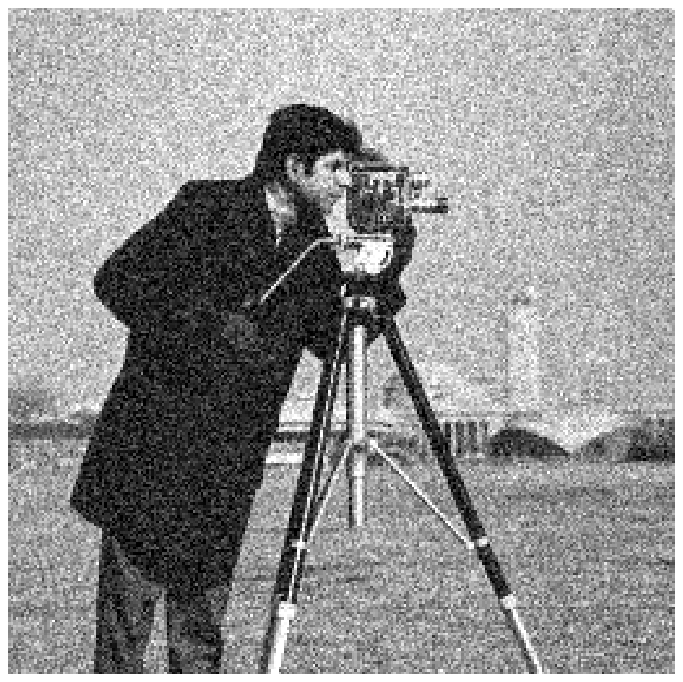}
 }
 \caption{Simulation results of image encryption with noise. (a)-(c) The original images with 0.03, 0.07 and 0.1 white Gaussian noise, respectively; (d)-(f) the corresponding retrieved noise images.
 }
 \label{fig:adaptive:noise}
\end{figure*}

\section{Conclusion}
In this report, a new formula for representing the complex-valued systems in matrix form was suggested, and the corresponding scheme to realize adaptive synchronization for a general class of complex-valued systems with fully unknown parameters was designed. The scheme was tested on the example of the complex-valued Rabinovich system that describes the parametric excitation of waves in a magneto-active plasma. A novel scheme for secure communication with improved cryptosystem based on the adaptive synchronization was proposed; the scheme is constructed to split the message and inject some bit of information signal into parameters modulation and the other bit into the transmitter states, which, in turn, complicates decryption task for possible intruders. Meanwhile, the information signal can be accurately retrieved at the receiver side by adaptive controllers and a decryption function. Different types of encrypted messages were considered for testing the robustness of the proposed scheme (e.g., plain text and gray images with diverse scales of white Gaussian noise).
\section*{Acknowledgement}
This work was supported by fundings from the Russian Science Foundation (Project 19-41-02002), from Saint Petersburg State University (PURE ID 75206671), and from the General Administration of Missions, Ministry of Higher Education of Egypt.

\bibliographystyle{elsarticle-num}
\bibliography{../../../../../Dropbox/bib/bib_full,../../../../../Dropbox/bib/bib-hidden,../../../../../Dropbox/bib/bib_pll,../../../../../Dropbox/bib/bib_nk,../../../../../Dropbox/bib/bib_leonov,../../../../../Dropbox/bib/genlorenz-bib,../../../../../Dropbox/bib/bib_stab,complex_lorenz}
\end{document}